\documentclass[12pt]{amsart}

\usepackage{amssymb}
\usepackage{amsmath}
\usepackage{amsthm}
\usepackage[all]{xy}

\theoremstyle{plain}
\newtheorem{thm}{Theorem}[section]
\newtheorem{prop}[thm]{Proposition}

\newtheorem{cor}[thm]{Corollary}

\theoremstyle{definition}
\newtheorem{defn}[thm]{Definition}

\theoremstyle{remark}
\newtheorem{rem}[thm]{Remark}

\DeclareMathOperator{\Spec}{Spec}
\DeclareMathOperator{\Proj}{Proj}

\DeclareMathOperator{\id}{id}

\def\N{\mathbb{N}}
\def\Z{\mathbb{Z}}
\def\Q{\mathbb{Q}}
\def\R{\mathbb{R}}
\def\C{\mathbb{C}}

\def\r+{\mathbb{R}_{\geq 0}}

\def\ep{\varepsilon}

\def\s'sp{s_{2}(P;\sigma)}
\def\r+{{\R}_{\geq 0}}
\def\P{\mathbb{P}}
\def\arw{\rightarrow}

\def\w.{W_{\bullet}}

\def\*c{\C^{\times}}
\def\s{\tilde{s}}

\newcommand{\calo}{\mathcal {O}}

\makeatletter
  
  \@addtoreset{equation}{section}
\makeatother

\begin{document}

\title{Algebro-geometric characterization of Cayley polytopes}
\author{Atsushi Ito}
\address{Department of Mathematics, Kyoto University, Kyoto 606-8502, Japan}
\email{aito@math.kyoto-u.ac.jp}

\begin{abstract}
In this paper,
we give an algebro-geometric characterization of Cayley polytopes.
As a special case,
we also characterize lattice polytopes with lattice width one
by using Seshadri constants.
\end{abstract}


\maketitle


\section{Introduction}\label{intro}

Let $P_0,\ldots,P_r$ be lattice polytopes in $\R^s$.
The Cayley sum $P_0*\cdots*P_r$
is defined to be the convex hull of
$(P_0 \times 0) \cup  (P_1 \times e_1) \cup \ldots \cup(P_r \times e_r)$
in $\R^s \times \R^r$ for the standard basis $e_1,\ldots,e_r$ of $\R^r$.

A lattice polytope $P \subset \R^n$
is said to be a Cayley polytope of length $r+1$,
if there exists an affine isomorphism $\Z^n \cong \Z^{n-r} \times \Z^r $
identifying $P$ with the Cayley sum $P_0*\cdots*P_r$
for some lattice polytopes $P_0,\ldots,P_r$ in $\R^{n-r}$.
In other words,
$P$ is a Cayley polytope of length $r+1$
if and only if $P$ is mapped onto a unimodular $r$-simplex
by a lattice projection $\R^n \arw \R^r$.

On the other hand,
an $n$-dimensional polarized toric variety $(X_P,L_P)$ is defined
for any lattice polytope $P \subset \R^n$ of dimension $n$.
The purpose of this paper is to characterize Cayley polytopes algebro-geometrically.
To state the main theorem of this paper,
we use the following notation.

\begin{defn}\label{def of planes}
Let $r$ be a positive integer.
A polarized variety $(X,L)$ is called an $r$-plane
if it is isomorphic to $(\P^r,\calo_{\P^r}(1))$.
To simplify notation,
we say $X$ is an $r$-plane
when the polarization $L$ is clear (e.g.\ a subvariety of a polarized variety).

Let $(X,L)$ be a polarized variety.
We say that $(X,L)$ is covered by $r$-planes,
if for any general point $p \in X$
there exists a subvariety $Z \subset X$ containing $p$ such that $(Z,L|_Z)$ is an $r$-plane.
When $r=1$,
we say that $(X,L)$ is covered by lines.
\end{defn}

If a lattice polytope $P \subset \R^n$ is a Cayley polytope of length $r+1$,
it is not difficult to see that $(X_P,L_P)$ is covered by $r$-planes.
The following theorem states that the converse holds.

\begin{thm}\label{thm1}
Let $P \subset \R^n$ be a lattice polytope of dimension $n$.
Then $P$ is a Cayley polytope of length $r+1$
if and only if $(X_P,L_P)$ is covered by $r$-planes.
\end{thm}

In \cite{CD},
Casagrande and Di Rocco investigate $\Q$-factorial toric varieties $X \subset \P^N$ which are covered by lines in detail.
In the case $X_P$ is $\Q$-factorial, $L_P$ is very ample, and $r=1$,
Theorem \ref{thm1} follows from \cite{CD}.
In Theorem \ref{thm1},
we do not need any assumption on
the singularities of $X_P$
nor the very ampleness of $L_P$ .

\vspace{2mm}

Cayley polytopes or Cayley sums are used to study dual defects
in \cite{CD},  \cite{CC}, \cite{DDP}, \cite{DN}, \cite{DS}, \cite{DR}, \cite{Es}, \cite{GKZ}, etc.
For example,
if $X_P$ is smooth for $P \subset \R^n$,
Di Rocco \cite{DR} proved that $X_P \subset \P^N$ embedded by $L_P$ has a positive dual defect $d$
(see Definition \ref{def_dual_defect} for the definition of dual defects)
if and only if $P$ is $\Z$-affine isomorphic to a Cayley sum $P_0 * \cdots * P_{(n+d)/2}$
such that the normal fans of the $P_i$ are the same and $n+d \geq 4$.
When $X_P$ is $\Q$-factorial and $L_P$ is very ample,
a similar result is shown in \cite{CD}.

Curran and Cattani \cite{CC} and Esterov \cite{Es} show a relation between dual defects and Cayley polytopes
without any assumption on the singularities.
Roughly,
they prove that if $X_P \subset \P^{N}$ embedded by $L_P$ has a positive dual defect,
$P$ is a Cayley polytope of length $2$.
As a corollary of Theorem \ref{thm1},
we can refine this result (see Corollary \ref{dual defect and cayley poly} for the details).

\vspace{3mm}
We investigate the case $r=1$ a little more.
For a polarized variety $(X,L)$,
we can define a positive number $\ep(X,L;1)$,
which is called the Seshadri constant of $(X,L)$ at a very general point.
This is an invariant which measures the positivity of $(X,L)$.
In the following theorem,
we characterize Cayley polytopes of length $2$
by using Seshadri constants.

\begin{thm}\label{thm2}
Let $P \subset \R^n$ be a lattice polytope of dimension $n$.
Then the following are equivalent.
\begin{itemize}
\item[i)] $P$ is a Cayley polytope of length $2$,
\item[ii)] $(X_P,L_P)$ is covered by lines,
\item[iii)] $\ep(X_P,L_P;1)=1$.
\end{itemize}
\end{thm}

In general,
it is very difficult to compute Seshadri constants.
Theorem \ref{thm2} gives an explicit description
for which lattice polytopes $P$
the Seshadri constant $\ep(X_P,L_P;1)$ is one.
\vspace{3mm}

This paper is organized as follows. In Section 2, we make some preliminaries.
In Section 3, we give the proof of Theorem \ref{thm1}.
In Section 4, we state a relation between Cayley polytopes and dual defects.
In Section 5, we prove Theorem \ref{thm2}.

\subsection*{Acknowledgments}
The author would like to express his gratitude to
Professor Benjamin Nill for kindly listening to
his naive idea
and informing him of the background of Cayley polytopes and references.
He would like to thank Professor Sandra Di Rocco
for valuable
suggestions and comments.
He is also grateful to his supervisor Professor Yujiro Kawamata
for giving him useful comments.
He is indebted to the anonymous referee for many helpful suggestions and comments.
In the first version of this paper,
we use some degenerations to prove the ``if part'' of Theorem \ref{thm1}.
The referee suggested considering some variant of Gr\"{o}bner degenerations. 
Although degenerations are not used in this version explicitly,
Gr\"{o}bner degenerations underlie in the proof.

The author was partially supported by
the Grant-in-Aid for JSPS fellows, No.\ 23-56182.

\section{Preliminaries}

\subsection{Notations and conventions}

We denote by $\N,\Z,\R$, and $\C$ the set of all
natural numbers, integers, real numbers, and complex numbers respectively.
In this paper, $\N$ contains $0$.
We denote $\C \setminus 0$ by $\*c$.

Let us denote by $e_1,\ldots,e_r$ the standard basis of $\Z^r$ or $\R^r$.
A \emph{lattice polytope} in $\R^n$ is the convex hull of a finite number of points in $\Z^n$.
The \emph{dimension} of a lattice polytope $P \subset \R^n$ is
the dimension of the affine space spanned by $P$.
For a subset $S$ in an $\R$-vector space,
we write $\Sigma(S)$ for
the closed convex cone spanned by $S$. 

Throughout this paper, we consider varieties or schemes over $\C$.
For a variety $X$, we say a property holds at a \emph{general} point of $X$
if it holds for all points in the complement of a proper algebraic subset.
For a variety $X$, we say a property holds at a \emph{very general} point of $X$
if it holds for all points in the complement of the union of countably many proper subvarieties.

\subsection{Cayley polytopes}

We say a linear map $\R^n \arw \R^r$ is a lattice projection
if it is induced from a surjective group homomorphism
$\Z^n \arw \Z^r$.

Let us recall the definitions of Cayley polytopes and lattice width.

\begin{defn}\label{def of Cayley}
Let $P$ be a lattice polytope in $\R^n$ and $r$ be a positive integer.
We say $P$ is a Cayley polytope of length $r+1$
if there exists a lattice projection $\R^n \arw \R^r$ which maps $P$ onto a unimodular $r$-simplex.
A unimodular $r$-simplex is a lattice polytope
which is identified with the convex hull of $0,e_1,\ldots,e_r$ in $\R^r$
by a $\Z$-affine translation.
\end{defn}

\begin{defn}\label{def of width}
Let $P$ be a lattice polytope in $\R^n$.
The lattice width of $P$
is the minimum of $\max_{u \in P}  \langle u,v \rangle - \min_{u \in P} \langle u,v \rangle$
over all non-zero integer linear forms $v$.
\end{defn}

\begin{rem}\label{cayley and width}
See \cite{BN} or \cite{DHNP} for other equivalent definitions of Cayley polytopes.
Note that $P$ has lattice width one
if and only if
$P$ is an $n$-dimensional Cayley polytope of length $2$.
\end{rem}

\subsection{Toric varieties}
In this subsection,
we recall notations about toric varieties
used in this paper.
We refer the reader to \cite{Fu} for a further treatment.

Let $P$ be a lattice polytope of dimension $n$ in $\R^n$.
Then we can define the polarized toric variety associated to $P$ as
\[
(X_P,L_P)= (\Proj \C[\Gamma_P], \calo(1)),
\]
where $\Gamma_P:=\Sigma(\{1\} \times P) \cap (\N \times \Z^n)$ is a subsemigroup of $\N \times \Z^n$.
We consider that $\Gamma_P$ is graded by $\N$,
that is,
the degree $k$ part of $\Gamma_P $ is $\Gamma_P \cap (\{k\} \times \Z^n)=\{k\} \times (kP \cap \Z^n) $. 
There exists a natural action on $X_P$ by the torus $(\*c)^n$.
We denote the maximal orbit in $X_P$ by $O_P = (\*c)^n.$

By definition,
a lattice point $u \in P \cap \Z^n$ corresponds to a global section $x^u \in H^0(X_P,L_P)$.
It is well known that such global sections form a basis of $H^0(X_P,L_P)$
and the linear system $|L_P|$ is base point free.
We denote by $\phi_P$ the morphism $X_P \arw \P^N$ defined by $|L_P|$,
where $N=\# (P \cap \Z^n)-1.$ 
Since $L_P$ is ample,
$\phi_P$ is a finite morphism.

\section{Proof of Theorem \ref{thm1}}

This section is devoted to the proof of Theorem \ref{thm1}.

\begin{proof}[Proof of Theorem \ref{thm1}]
Assume that $P$ is a Cayley polytope of length $r+1$
and take a lattice projection $\pi : \R^n \arw \R^r$
such that $\pi(P) $ is a unimodular $r$-simplex.
Restricting $\id_{\N} \times \pi : \N \times \Z^n \arw \N \times \Z^r$ on $\Gamma_P$,
we obtain a homomorphism of graded semigroups
\[
\psi : \Gamma_P \arw \Gamma_{\pi(P)} \ : \ (k,u) \mapsto (k,\pi(u)).
\]
Since $\pi(P) $ is a unimodular $r$-simplex,
any $u' \in \pi(P) \cap \Z^r$ is the image of some vertex of $P$ by $\pi$.
Hence the restriction of $\psi $ to the degree $1$ part
\[
\psi |_{\Gamma_P \cap (\{1\} \times \Z^n)} : \Gamma_P \cap (\{1\} \times \Z^n) \arw \Gamma_{\pi(P)} \cap (\{1\} \times \Z^r)
\]
is surjective.
Since $ \Gamma_{\pi(P)}$ is generated by elements of degree $1$,
$\psi$ is surjective.
Hence $\psi$ induces an embedding $\iota: X_{\pi(P)} \hookrightarrow X_P$ such that $\iota^* L_P = L_{\pi(P)} $.
Since $\pi(P)$ is a unimodular $r$-simplex,
$(X_{\pi(P)}, L_{\pi(P)})$ is an $r$-plane.
Hence $\iota(X_{\pi(P)})$ is an $r$-plane on $(X_P,L_P)$.
Since $\psi$ is the restriction of $\id_\N \times \pi : \N \times \Z^n \arw \N \times \Z^r$,
$\iota |_{O_{\pi(P)}}$ is nothing but the morphism
\[
O_{\pi(P)} = \Spec \C[\Z^r]  \hookrightarrow O_{P}= \Spec \C[\Z^n] \subset X_P
\]
defined by $\C[\Z^n] \arw \C[\Z^r] : x^u \arw x^{\pi(u)}$.
In particular,
$ \iota(X_{\pi(P)}) \cap O_P \not= \emptyset$.
Since $\iota(X_{\pi(P)})$ is an $r$-plane,
$(X_P,L_P)$ is covered by $r$-planes by the torus action.

\vspace{3mm}

To prove the converse,
assume that $(X_P,L_P)$ is covered by $r$-planes.
Then there exists an $r$-plane $Z \subset X_P$ such that
$Z \cap O_P \not = \emptyset$.
Consider the natural homomorphism of graded $\C$-algebras
\[
 \alpha : \bigoplus_{k \in \N} H^0(X_P, k L_P) \arw  \bigoplus_{k \in \N} H^0(Z , k L_P |_Z) .
\]
If $\alpha |_{H^0(X_P, L_P)} :  H^0(X_P,L_P) \arw H^0(Z, L_P |_Z)$ is not surjective,
$|L_P|$ has a base point on $Z$ since $(Z, L_P |_Z) \cong (\P^r, \calo_{\P^r}(1))$,
which is a contradiction.
Hence
\[
\alpha |_{H^0(X_P, L_P)} :  H^0(X_P,L_P) = \bigoplus_{u \in P \cap \Z^n} \C x^u  \arw H^0(Z, L_P |_Z)
\]
is surjective.
Thus there exist lattice points $u_0,\cdots,u_r \in P \cap \Z^n$
such that
$\alpha(x^{u_0}), \cdots, \alpha(x^{u_r})$ form a basis of $H^0(Z,L_P |_Z)$.
Set $X_i = \alpha (x^{u_i}) \in H^0(Z,L_P|_Z)$.
Since $(Z,L_P|_Z)$ is an $r$-plane, it holds that
\[
\bigoplus_{k \in \N} H^0(Z , k L_P |_Z) =\C[X_0,\cdots, X_r] .
\]
Since $\bigoplus_{k \in \N} H^0(X_P, k L_P) =\C[\Gamma_P] $,
$\alpha$ can be written as
\[
\alpha : \C[\Gamma_P]  \arw \C[X_0,\ldots, X_r].
\]
Let $\Delta_r \subset \R^r$ be the unimodular $r$-simplex whose vertices are $0, e_1,\cdots,e_r$.
We define a homomorphism of graded semigroups
\[
\beta : \Gamma_P \arw \Gamma_{\Delta_r}
\]
as follows.

For $(k,u) \in \Gamma_P$,
let $x^{(k,u)} \in \C[\Gamma_P]$ be the corresponding element.
Since $\alpha (x^{(k,u)})$ is a homogeneous polynomial of degree $k$ in $\C[X_0,\ldots, X_r]$,
\[
\alpha (x^{(k,u)}) = \sum_{ I=(i_1,\cdots,i_r) \in k \Delta_r \cap \Z^r} c_I X_0^{k - |I|}X_1^{i_1} \cdots X_r^{i_r}
\]
holds for some $c_I \in \C$,
where $|I| := i_1 + \cdots + i_r$.
By the identification of $\bigoplus_{k \in \N} H^0(X_P, k L_P)$ and $\C[\Gamma_P]$,
we identify $x^u \in H^0(X_P,k L_P)$ with $x^{(k,u)} \in \C[\Gamma_P]$.
Therefore,
if $ \alpha (x^{(k,u)}) =0$,
$Z$ is contained in the effective divisor defined by $ x^u \in H^0(X_P, k L_P)$,
which contradicts $Z \cap O_P \not = \emptyset$.
Thus $\alpha (x^{(k,u)}) \not=0 $ holds,
i.e.\
$\{ I \in k \Delta_r \cap \Z^r \ | \ c_I \not= 0 \} $ is not empty.
Hence we can define
\[
\beta(k,u) := (k, \min \{ I \in k \Delta_r \cap \Z^r \ | \ c_I \not= 0 \}) \in \Gamma_{\Delta_r},
\]
where the minimum is taken with respect to a fixed monomial order on $\N^r$ (e.g.\ the lexicographic order).

By definition,
\begin{align}\label{eq_1}
\beta(1,u) \in \{1\} \times (\Delta_r \cap \Z^r) 
\end{align}
holds
for any $u \in P \cap \Z^n$.
Since $\alpha(x^{(1,u_i)}) =X_i$ for $0 \leq i \leq r$,
we have 
\begin{equation}\label{eq_2}
\beta(1,u_i)=
\left\{
\begin{array}{ll}
(1, 0) &  \mbox{if \ \ \ \ $i = 0$}  \ \\
(1,e_i) &  \mbox{if \ $1 \leq i\leq r$} \ .   
\end{array}
\right .
\end{equation}

Since $\alpha$ is a ring homomorphism,
$\beta$ is a semigroup homomorphism.
By definition,
$\beta$ preserves the gradings.
Hence $\beta$ is a homomorphism of graded semigroups.

Since $\Gamma_P$ and $\Gamma_{\Delta_r}$ generate $ \Z \times \Z^n$ and $\Z \times \Z^r$ as groups respectively,
$\beta$ uniquely extends to a group homomorphism $ \Z \times \Z^n \arw \Z \times \Z^r$.
By abuse of notation,
we denote the group homomorphism by the same letter $\beta$.
By restricting $\beta$ on $\{0\} \times \Z^n$,
we obtain a group homomorphism $\pi : \Z^n \arw \Z^r$.
In other words,
$\pi$ is defined by $\beta (0, u) = (0, \pi(u))$ for $u \in \Z^n$.
Since $\beta$ and $\pi$ are group homomorphisms,
it holds that
\[
\beta (1,u) -\beta (1,u_0) = \beta (0, u-u_0) =(0, \pi (u) - \pi(u_0))
\]
for $u \in \Z^n$.
On the other hand,
\[
\beta (1,u) -\beta (1,u_0) = \beta (1,u) -(1,0) \in \{0\} \times (\Delta_r \cap \Z^r)
\]
holds for $u \in P \cap \Z^n$ by (\ref{eq_1}), (\ref{eq_2}).
Hence we have $\pi (u) - \pi(u_0) \in  \Delta_r \cap \Z^r$ for $u \in P \cap \Z^n$.
In particular,
if $u=u_i$ for $1 \leq i \leq r$,
\[
\beta (1,u_i) -\beta (1,u_0) = (1,e_i) -(1,0) =(0,e_i),
\]
hence $\pi(u_i) -\pi(u_0)=e_i$.
Thus $\pi$ is surjective
and the induced lattice projection $\R^n \arw \R^r$ maps $P$ onto the unimodular $r$-simplex $\Delta_r + \pi(u_0)$,
which is the parallel translation of $\Delta_r$ by $\pi(u_0)$.
Hence $P$ is a Cayley polytope of length $r+1$.
\end{proof}


\section{Dual defects}

In this section,
we see a relation between Cayley polytopes and dual defects.

\begin{defn}\label{def_dual_defect}
Let $X \subset \P^N$ be a projective variety.
The dual variety $X^*$ of $X$ is the closure of all points $H \in  (\P^N)^{\vee }$ 
such that as a hyperplane $H$ contains the tangent space $T_{X,p}$ for some smooth point $p \in X$,
where $(\P^N)^{\vee }$ is the dual projective space.
A variety $X$ in $\P^N$ is said to be dual defective
if the dimension of $X^*$ is less than $N-1$.
The dual defect of $X$ is the natural number $N-1 - \dim X^*$.
\end{defn}

By \cite{GKZ}, \cite{DFS}, \cite{MT}, etc.,
we have formulas to compute the dual defects or the degrees of the dual varieties of toric varieties.
As stated in Introduction,
a simple description of the dual defects of toric varieties is obtained by using Cayley polytopes
for the smooth case in \cite{DR}, \cite{DDP}, and \cite{DN}, 
and for the $\Q$-factorial case in \cite{CD}.

\vspace{2mm}
As an easy corollary of Theorem \ref{thm1},
we obtain a sufficient condition such that $P$ is a Cayley polytope
by using dual defects.
This is a generalization of a result proved in \cite{CC} or \cite{Es},
which is the case $r=1$ of the following.

\begin{cor}\label{dual defect and cayley poly}
Let $A \subset \Z^n$ be a finite set and let $P \subset \R^n$ be the convex hull of $A$.
Assume that 
the lattice $\Z^n$ is spanned by $\{ u - u' \, | \, u,u' \in  A\}$.
Let $\phi : X_P \arw \P^{\# A-1}$ be the morphism defined by the linear system $|V_A|$,
where $V_A =  \bigoplus_{u \in A} \C x^u  \subset H^0(X_P,L_P)$.
If the dual defect of the image $\phi (X_P) \subset \P^N$ is a positive integer $r$,
$ P$ is a Cayley polytope of length $r+1$.
In particular,
$P$ has lattice width one.
\end{cor}

\begin{proof}
It is known that if a projective variety $X \subset \P^N$ has dual defect $r$,
$X$ is covered by $r$-planes
(cf.\ \cite[Theorem 1.18]{Te} for example).
Thus if the dual defect of $\phi (X_P)$ is $r$,
then $\phi (X_P)$ is covered by $r$-planes.
The assumption that
$\Z^n$ is spanned by $\{ u - u' \, | \, u,u' \in A\}$
means that
$\phi : X_P \arw \phi (X_P)$ is birational.
Since $\phi (X_P)$ is covered by $r$-planes
and $\phi : X_P \arw \phi (X_P) $ is birational and finite,
$(X_P, \phi^*(\calo(1)))=(X_P,L_P)$ is also covered by $r$-planes.
Therefore $P $ is a Cayley polytope of length $r+1$ by Theorem \ref{thm1}.
\end{proof}

\begin{rem}
(1) In Corollary \ref{dual defect and cayley poly},
the assumption that $\Z^n$ is spanned by $\{ u - u' \, | \, u,u' \in A\}$
is necessary.
For example,
let $P \subset \R^3$ be the convex hull of $A:=\{ (0,0,0),(1,1,0),(1,0,1), (0,1,1) \}$.
Then the image $\phi (X_P) \subset \P^3$ is $\P^3$,
hence the dual defect of $\phi (X_P)$ is $3$.
But $P$ is not a Cayley polytope of length $4$.\\
(2) The converse of Corollary \ref{dual defect and cayley poly} does not hold.
For example, let $P \subset \R^2$ be the convex hull of $A:=\{(0,0),(1,0),(0,1),(1,1)\}$.
Then $P$ is a Cayley polytope of length $2$,
but $\phi (X_P)=X_P \subset \P^3$ is a smooth quadric surface,
which is not dual defective.
\end{rem}


\section{Lattice width one and Seshadri constants}

In \cite{De},
Demailly defined Seshadri constants,
which measure the positivity of line bundles.
In this section,
we characterize lattice polytopes with lattice width one
by using Seshadri constants.
We refer the reader to Chapter 5 of \cite{La} for a detailed treatment of Seshadri constants.

\vspace{2mm}
Recall the definition of Seshadri constants.

\begin{defn}
Let $(X,L)$ be a polarized variety 
and $p $ be
a point in $X$.
The Seshadri constant $\ep(X,L;p)$ of $L$ at $p$ is defined to be
\[
\ep(X,L;p)=\inf_C \dfrac{C.L}{m_p(C)} ,
\]
where $C$ varies over all curves on $X$ containing $p$
and $m_p(C)$ is the multiplicity of $C$ at $p$.
It is well known that
\[
\ep(X,L;p)= \max \{ s >0 \, | \, \mu^*L-s E \  \text{is nef} \, \} 
\]
holds, where $\mu$ is the blow-up at $p$ and $E$ is the exceptional divisor.

It is known that $\ep(X,L;p)$ is constant for very general $p$.
Thus we can define $\ep(X,L;1)$ to be
\[
\ep(X,L;1)=\ep(X,L;p) 
\]
for very general $p \in X$.
\end{defn}

Explicit computations of Seshadri constants sometimes
give interesting geometric consequences,
but
it is very difficult to compute Seshadri constants in general.
Even in toric cases, Seshadri constants are computed only at torus invariant points \cite{DR1}, \cite{B+}.

When $L$ is very ample,
it follows from Lemma 2.2 in \cite{Ch} that
$\ep(X,L;1)=1$ if and only if 
$(X,L)$ is covered by lines.
We generalize this to the case when $L$ is base point free.

\begin{prop}\label{sc and lines}
Let $(X,L)$ be a polarized variety
and assume that the linear system $|L|$ is base point free.
Then $\ep(X,L;1)=1$
if and only if $(X,L)$ is covered by lines.
\end{prop}

\begin{proof}
If $(X,L)$ is covered by lines,
$\ep(X,L;1) \leq 1$ holds by the definition of Seshadri constants.
On the other hand,
$\ep(X,L;1) \geq 1$ holds since $L$ is base point free (cf.\ \cite[Example 5.1.18]{La}).
Thus we have $\ep(X,L;1)= 1$.

\vspace{2mm}

To show the converse,
assume $\ep(X,L;1)=1$.
Since the locus
\[
\{ p \in X \, | \, \text{there exists a line on }  (X,L) \text{ containing } p  \}
\]
is Zariski closed in $X$,
it suffices to show that there exists a line on $X$ containing $p$ for \textit{very general} $p \in X$.
Let $\phi :X \rightarrow \P^N$ be the morphism defined by $|L|$,
and let $d$ be the degree of the finite morphism $\phi:X \rightarrow \phi(X)$.
Fix a very general point $p \in X$ and  set $q=\phi(p) \in \P^N$.
Then $\phi^{-1}(q)=\{p_1,\ldots,p_d\}$ is a set of $d$ points in $X$,
where $p=p_1$.

We consider the following diagram:
\[\xymatrix{
& \widetilde{X} \ar[dl]_(0.45){\pi} \ar[d]_(0.45){\widetilde{\phi}} \ar@/^/[ddr]^(0.45)f & \\
 X \ar[d]_(0.45){\phi}  & \widetilde{\P}^N \ar[dl]^(0.42){\pi'} \ar[dr]_(0.42){f'} &  \\
  \P^N & & \P^{N-1}   .\\
}\]
In the above diagram,
$\pi:\widetilde{X} \rightarrow X$ and $ \pi':\widetilde{\P}^N \rightarrow \P^N$
are the blow-ups along $\{p_1,\ldots,p_d\}$ and $q$ respectively.
Then $\phi, \pi$, and $\pi'$ induce a finite morphism
$\widetilde{\phi}: \widetilde{X} \rightarrow \widetilde{\P}^N$.
Let $E_1,\ldots,E_d$ and $E$ be the exceptional divisors over $p_1,\ldots,p_d$ and $q$ respectively.
The line bundle $\pi'^{*} \calo_{\P^N}(1)-E$ is base point free
and induces a morphism
$f':\widetilde{\P}^N \rightarrow \P^{N-1}$.
Set $f=f' \circ \widetilde{\phi} : \widetilde{X} \arw \P^{N-1}$.
Thus $f^* \calo_{\P^{N-1}}(1)=\pi^*L-\sum_{i=1}^d E_i$ holds.

By the assumption that $\ep(X,L;1)=1$ and $p$ is very general,
$\pi^*L - E_1 $ is nef but not ample.
Thus there exists a subvariety $\widetilde{Z}$ in $\widetilde{X}$
such that $\widetilde{Z}.(\pi^*L - E_1)^{\dim \widetilde{Z}}=0$ by Kleiman's criterion.
Furthermore, $\widetilde{Z}. \bigl(\pi^*L - \sum_{i=1}^d E_i \bigr)^{\dim \widetilde{Z}} \geq 0$
by the base point freeness of $\pi^*L - \sum_{i=1}^d E_i$.
Hence we have
\begin{align*}
0 &\leq \widetilde{Z}.\bigl(\pi^*L - \sum_{i=1}^d E_i \bigr)^{\dim \widetilde{Z}}\\
   &= Z.L^{\dim Z} - \sum_{i=1}^d m_{p_i}(Z)\\
   &\leq  Z.L^{\dim Z} - m_{p_1}(Z)\\
   &= \widetilde{Z}.(\pi^*L - E_1)^{\dim \widetilde{Z}} =0,
\end{align*}
where $Z=\pi(\widetilde{Z})$ is the image of $\widetilde{Z}$ by $\pi$
and $m_{p_i}(Z)$ is the multiplicity of $Z$ at $p_i$.
Thus we obtain 
\[
0 = \widetilde{Z}.\bigl(\pi^*L - \sum_{i=1}^d E_i \bigr)^{\dim \widetilde{Z}} =\widetilde{Z}.(\pi^*L - E_1)^{\dim \widetilde{Z}}
\]
and $m_{p_i}(Z)=0$ for  $i \not = 1$.
This means $p_i \not \in Z$,
or equivalently $\widetilde{Z} \cap E_i = \emptyset$ for  $i \not = 1$.
The equality $\widetilde{Z}.\bigl(\pi^*L - \sum_{i=1}^d E_i \bigr)^{\dim \widetilde{Z}}=0$
implies $\dim f(\widetilde{Z}) < \dim \widetilde{Z}$.
Hence there exists a curve $\widetilde{C} \subset \widetilde{Z}$
such that $f(\widetilde{C})$ is a point.
Set $C=\pi(\widetilde{C})$ to be the image of $\widetilde{C}$ in $X$.
Since $\widetilde{\phi}$ is a finite morphism,
$\widetilde{\phi}(\widetilde{C})$ is a curve on $\widetilde{\P}^N$
which is contracted by $f'$.
Thus $\phi(C) =\pi'(\widetilde{\phi}(\widetilde{C}))$ is a line on $\P^N$ containing $q$.
Since $\phi : X \arw \phi(X)$ is \'etale at $p \in X$
and $\phi(C) \cong \P^1$ is smooth at $q=\phi(p)$,
$C$ is also smooth at $p$.
Since $f(\widetilde{C})$ is a point,
$\widetilde{C}.\bigl(\pi^*L-\sum_{i=1}^d E_i \bigr)=0$ holds.
Furthermore, $\widetilde{C} \cap E_i \subset \widetilde{Z} \cap E_i = \emptyset$ for $i \not = 1$.
Thus $0=\widetilde{C}.(\pi^*L-E_1)=C.L - m_p(C)$ holds.
Hence
we have
\[
1= m_p(C) = C.L = \phi_*(C) . \calo_{\P^N}(1)= \deg(\phi |_C : C \rightarrow \phi(C)).
\]
Thus $\phi |_C : C \rightarrow \phi(C) \cong \P^1$ is an isomorphism
and $C.L=1$,
i.e.\ $C$ is a line on $X$ containing $p$.
Hence $(X,L)$ is covered by lines.
\end{proof}

\begin{rem}
The assumption that $|L|$ is base point free is necessary in Proposition \ref{sc and lines}.
For example,
let $(S,L)$ be a polarized smooth surface of general type such that $L^2=1$,
e.g.\ $S$ is a Godeaux surface and $L$ is the canonical divisor $K_S$.
Then $\ep(S,L;1)=1$ holds by \cite{EL} and the assumption $L^2=1$.
On the other hand,
$S$ is not covered by lines because $S$ is of general type.
\end{rem}

By Theorem \ref{thm1} and Proposition \ref{sc and lines} ,
we obtain the following corollary.

\begin{cor}[=Theorem \ref{thm2}]\label{width 1 and sc}
Let $P \subset \R^n$ be a lattice polytope of dimension $n$.
Then the following are equivalent.
\begin{itemize}
\item[i)] $P$ has lattice width one,
\item[ii)] $(X_P,L_P)$ is covered by lines,
\item[iii)] $\ep(X_P,L_P;1)=1$.
\end{itemize}
\end{cor}

\begin{proof}
i) $\Leftrightarrow$ ii) follows from Theorem \ref{thm1} and Remark \ref{cayley and width}.
ii) $\Leftrightarrow$ iii) follows from Proposition \ref{sc and lines}.
\end{proof}

As stated in Introduction,
this corollary tells us for which $P$
the Seshadri constant $\ep(X_P,L_P;1)$ is one.
We note that Nakamaye \cite{Na} gives an explicit description
for which polarized abelian varieties $(A,L)$
the Seshadri constant $\ep(A,L;1)$ is one.

\end{document}